%% file: main.tex
\documentclass[a4paper,oneside,11pt]{article}

\usepackage[english]{babel} 
\usepackage[T1]{fontenc}
\usepackage[utf8]{inputenc}
\usepackage{lmodern} 
\usepackage[numbers]{natbib}
\usepackage{hyphenat}
\usepackage{amsmath}
\usepackage{amsthm}

\begingroup
    \makeatletter
    \@for\theoremstyle:=definition,remark,plain\do{%
        \expandafter\g@addto@macro\csname th@\theoremstyle\endcsname{%
            \addtolength\thm@preskip\parskip
            }%
        }
\endgroup

\usepackage{amsfonts}
\newtheorem{teo}{Theorem}[section]
\newtheorem{prop}[teo]{Proposition}
\newtheorem{lemma}[teo]{Lemma}
\newtheorem{re}[teo]{Remark}
\newtheorem{definicao}[teo]{Definition}
\newtheorem{cor}[teo]{Corollary}

\numberwithin{equation}{section}
\newcommand{\Real}{\mathbb R}

\usepackage[plainpages=false,pdfpagelabels,pdfusetitle,pdfdisplaydoctitle, pdfstartpage=1, pdfstartview={{XYZ null null 1.00}}]{hyperref}
\usepackage[top=3cm, bottom=2cm, left=3cm, right=2cm]{geometry}

\author{\textsc{Luccas Campos and Mykael Cardoso}}

\usepackage{parskip}
\usepackage{mathtools}
\mathtoolsset{showonlyrefs}
\title{Blow up and scattering criteria above the threshold for the focusing inhomogeneous nonlinear Schrödinger equation}
\author{Luccas Campos and Mykael Cardoso}
\date{}
\begin{document}

\maketitle
\begin{abstract}\noindent
We consider the inhomogeneous nonlinear Schrödinger equation (INLS) in $\mathbb{R}^N$
$$i \partial_t u + \Delta u + |x|^{-b} |u|^{p-1}u = 0,$$
with finite-variance initial data $u_0 \in H^1(\Real^N)$. We extend the dichotomy between scattering and blow-up for solutions above the mass-energy threshold (and with arbitrarily large energy). We also show other two blow-up criteria, wich are valid in any mass-supercritical setting, given there is local well-posedness.
\end{abstract}
\input{introduction}
\input{bound}

\input{quadratic}

\input{scattering}

\input{blow}

\bibliography{biblio}
\bibliographystyle{abbrvnat}
\newcommand{\Addresses}{{
  \bigskip
  \footnotesize

  L. Campos, \textsc{Department of Mathematics, UFMG, Brazil}\par\nopagebreak
  \textit{E-mail address:} \texttt{luccasccampos@gmail.com}

  \medskip

  M. Cardoso, \textsc{Department of Mathematics, UFMG, Brazil;}
  \textsc{Department of Mathematics, UFPI, Brazil}\par\nopagebreak
  \textit{E-mail address:} \texttt{mykael@ufpi.edu.br}

}}
\setlength{\parskip}{0pt}
\Addresses
\end{document}

%% file: introduction.tex
\section{Introduction}
We consider the initial value problem associated to the inhomogeneous nonlinear Schrödinger equation (INLS):
\begin{equation}\label{PVI}
\left\{
\begin{array}{l}
i\partial_t u +\triangle u +|x|^{-b}|u|^{p-1}u=0,\,\,\,\,t>0, \,\, x\in \mathbb{R}^N,\\
u(\cdot,0)=u_0\in H^{1}(\mathbb{R}^N).
\end{array}
\right.
\end{equation}

 This model arises naturally as a limiting problem in nonlinear optics for the propagation of laser beams.
The case $b = 0$ is the classical nonlinear Schrödinger equation (NLS), extensively studied in recent years (see Sulem-Sulem \cite{Sulem}, Bourgain \cite{Bo99}, Cazenave \cite{cazenave}, Linares-Ponce \cite{LiPo15}, Fibich \cite{Fi15} and the references therein).

The lower Sobolev index where one can expect well-posedness for this model is given by scaling. If $u(x, t)$ is a solution to \eqref{PVI}, so is $u_{\lambda}(x, t)=\lambda^{\frac{2-b}{p-1}}u(\lambda x,\lambda^2 t)$, with
initial data $u_{0,\lambda}(x),$ for all $\lambda > 0$. Computing the homogeneous Sobolev norm,
we get
$$\|u_{0,\lambda}\|_{\dot{H}^{s}}=\lambda^{s-\frac{N}{2}+\frac{2-b}{p-1}}\|u_0\|_{\dot{H}^{s}}.$$
Thus, the scale-invariant Sobolev norm is $\dot{H}^{s_c}(\Real^N)$, where  $$s_c=\frac{N}{2}-\frac{2-b}{p-1}$$ is called the \textit{critical Sobolev index}.
 
In this paper, we are interested in the case $s_c > 0$, known as \textit{mass-supercritical}. Rewriting this condition in terms of $p$, we obtain
\begin{equation}\label{sc}
p >1+ \frac{2(2-b)}{N}.
\end{equation}

The local well-posedness for the INLS equation was first studied by Genoud-Stuart in \cite{g_8} (see also Genoud \cite{g_6}) by the abstract theory of Cazenave \cite{cazenave}, without relying on Strichartz type inequalities. They analyzed the IVP \eqref{PVI} in the sense of distributions, that is, $i\partial_tu+\Delta u+|x|^{-b}|u|^{p-1}u=0 $ in $H^{-1}(\mathbb{R}^N)$ and showed, with $0 < b < 2$, it is well-posed
\begin{itemize}
        \item[-] locally if $1 < p < p^*_b$ ($s_c < 1$);
        \item[-] globally for any initial data in $H^1(\Real^{N})$ if $p <1+ \frac{2(2-b)}{N}$ ($s_c < 0$);
        \item[-] globally for sufficiently small initial data if $1+\frac{2(2-b)}{N}  \leq p < p_b^*$ ($0 \leq s_c < 1$),
\end{itemize}
where
\begin{equation}\label{p_cond}
p^*_b = \begin{cases} \infty, & N \leq 2,  \\
1+\frac{2(2-b)}{N-2}, & N \geq 3. \\
\end{cases}
\end{equation}
More recently, Guzmán \cite{Boa} established local well-posedness of the INLS in $H^s(\Real^{N})$ based on Strichartz estimates. In particular, setting 
$$
\tilde{2} = \begin{cases} \frac{N}{3}, & N \leq 3  \\
2, & N \geq 4, \\
\end{cases}
$$
he proved that, for $N \geq 2$, $1 < p < p_b^*$ and $0 < b < \tilde{2}$, the initial value problem \eqref{PVI} is locally well-posed in $H^1(\Real^N)$. Dinh \cite{Boa_Dinh} improved Guzmán's results in dimensions $N=2$ (for $0<b<1$ and $0 < p < p_b^*)$ and  $N = 3$ (for $0 < b < \frac{3}{2}$ and $0 < p <1+ \frac{6-4b}{2b-1}$). Note that the results of Guzmán \cite{Boa} and Dinh \cite{Boa_Dinh} do not treat the case $N = 1$, and the ranges of $b$ are more restricted than those in the results of Genoud-Stuart \cite{g_8}. However, Guzmán and Dinh give more detail information on the solutions, showing that there exists $T(\|u_0\|_{H^1})>0$ such that $u \in L^q\left([-T,T];L^r(\Real^N)\right)$ for any $L^2$-admissible pair $(q,r)$ satisfying
\begin{equation*}
\frac{2}{q}=\frac{N}{2}-\frac{N}{r},
\end{equation*}
where 
\begin{equation*}
\left\{\begin{array}{cl}
2\leq & r  \leq \frac{2N}{N-2}\hspace{0.5cm}\textnormal{if}\;\;\;  N\geq 3,\\
2 \leq  & r < +\infty\;  \hspace{0.5cm}\textnormal{if}\;\; \;N=2,\\
2 \leq & r \leq + \infty\;  \hspace{0.5cm}\textnormal{if}\;\;\;N=1.
\end{array}\right.
\end{equation*}

The solutions to \eqref{PVI} have the following conserved quantities
\begin{equation}\label{mass}
M\left[u(t) \right] = \int |u(t)|^2 dx = M[u_0],
\end{equation}
\begin{equation}\label{energy}
E\left[u(t) \right] = \frac{1}{2}\int |\nabla u(t)|^2 dx - \frac{1}{p+1} \int |x|^{-b}|u(t)|^{p+1} dx = E[u_0].
\end{equation}

The blow-up theory in the INLS equation is related to the concept of \textit{ground state}, which is the unique positive radial solution of the elliptic problem
\begin{equation}
\Delta Q - Q + |x|^{-b}|Q|^{p-1} Q = 0.
\label{ground_state}
\end{equation}

The existence of the ground state is proved by Genoud-Stuart \cite{g_5, g_8} for dimension $N \geq 2$, and by Genoud \cite{g_6} for $N = 1$. Uniqueness was proved in dimension $N \geq 3$ by Yanagida \cite{g_19} (see also Genoud \cite{g_5}), in dimension $N = 2$ by Genoud \cite{g_7} and in dimension $N = 1$ by Toland \cite{g_16}. The existence and uniqueness hold for $0 < b < \tilde{2}$ and $0 < \sigma < \sigma^*_b$.

The ground state satisfies the following Pohozaev's identities (see relations (1.9)-(1.10) in Farah\cite{Farah_well})\begin{equation}\label{pohozaev}
\|\nabla Q\|^{2}_{L^2}=\frac{N(p-1)+2b}{2(p+1)}\int|x|^{-b}|Q|^{p+1}\,dx,
\end{equation}
and
\begin{align}\label{EQ}
E[Q]=\frac{(p-1)s_c}{2(p+1)}\int|x|^{-b}|Q|^{p+1}\,dx.
\end{align}
Genoud \cite{Genoud1} and Farah \cite{Farah_well} proved the following sharp Gagliardo-Nirenberg inequality, valid for $0 \leq s_c < 1$ and $0<b<\min\{2,N\}$\begin{equation}\label{GN1}
\int_{\mathbb{R}^N}|x|^{-b}|f(x)|^{p+1}\,dx\leq C_{p,N}\|\nabla f\|_{L^2(\mathbb{R}^N)}^{\frac{N(p-1)+2b}{2}}\|f\|_{L^2(\mathbb{R}^N)}^{p+1-\frac{(N(p-1)+2b)}{2}},
\end{equation}
where $C_{p,N}>0$ is the sharp constant. More precisely,
\begin{align}\label{cnstGN}
C_{p,N}=\left(\frac{2(p+1)}{N(p-1)+2b}\right)^{\frac{N(p-1)+2b}{4}}\frac{\left(\displaystyle\int|x|^{-b}|Q|^{p+1}\right)^{1-\frac{N(p-1)+2b}{4}}}{\|Q\|^{p+1-\frac{N(p-1)+2b}{2}}_{L^2(\mathbb{R^N})}}.
\end{align}
This inequality can be seen as an extension to the case $b > 0$ of the classical Gagliardo-Nirenberg inequality. It is also an extension of the inequality obtained by Genoud \cite{Genoud1}, who showed its validity for $p =1+ \frac{2(2-b)}{N}$.

If $u$ is a solution to \eqref{PVI} and $u_0 \in \Sigma = \left\{ f \in H^1(\mathbb{R}^N) ; |x|f \in L^2(\mathbb{R}^N)\right\}$, we define its variance at time $t$ as
\begin{equation}
V(t)=\int|x|^2|u(x,t)|^2\,dx.
\end{equation}

The variance satisfies the \textit{virial identities} (see Farah \cite[Proposition 4.1]{Farah_well})
\begin{align}\label{dvar}
V_t(t)&=
4 \text{ Im}\int x\cdot \nabla u(x,t) \overline{u}(x,t)\,dx
\end{align}
and
\begin{align}\label{d2var}
V_{tt}(t)&=
4(N(p-1)+2b)E[u]-2(N(p-1)+2b-4)\|\nabla u\|^{2}_{L^2(\Real^N)}.
\end{align}

From this identity, we immediately see that, if $u_0 \in \Sigma$, $p >1+ \frac{2(2-b)}{N}$ and $E\left[u_0 \right] < 0$, then the graph of $t \mapsto \int |x|^2 |u|^2$ lies below an inverted parabola, which becomes negative in finite time. Therefore, the solution cannot exist globally and blows up in finite time. Recently, \citet{Blowup_Dinh} extended this result to the radial case, and to the case $N =\ 1$ without symmetry or decaying assumptions.
\subsection{Dichotomy above the mass-energy threshold}

An important scale-invariant quantity is $M[u_0]^{1-s_c}E[u_0]^{s_c}$, which we normalize (for $0 < s_c < 1$) as

\begin{equation}\label{massenergy}
\mathcal{ME}[u] = \mathcal{ME}[u_0]\ = \frac{ M[u_0]^\frac{1-s_c}{s_c}E[u_0]}{M[Q]^\frac{1-s_c}{s_c}E[Q]}
\end{equation}

and call it the \textit{mass-energy}. Other useful scale-invariant quantities are the \textit{mass-potential-energy}:
\begin{equation}\label{mp}
\mathcal{MP}[u(t)]\ = \frac{ M[u_0]^\frac{1-s_c}{s_c}\displaystyle\int|x|^{-b}|u(t)|^{p+1}}{M[Q]^\frac{1-s_c}{s_c}\displaystyle\int|x|^{-b}|Q|^{p+1}}
\end{equation}

and the \textit{mass-kinetical-energy}

\begin{equation}\label{mk}
\mathcal{MK}[u(t)]=       \frac{ M[u_0]^\frac{1-s_c}{s_c}\displaystyle\int |\nabla u(t)|^2}{M[Q]^\frac{1-s_c}{s_c}\displaystyle\int |\nabla Q|^2}.
\end{equation}

In previous works, \citet{FG_scat} and \citet{Blowup_Dinh} studied the global behavior of solutions to \eqref{PVI} below the \textit{mass-energy threshold}, i.e, in the case $\mathcal{ME}[u_0] < 1$. They proved a dichotomy between blow-up and scattering, depending on the quantity $\mathcal{MK}[u_0]$.

 We summarize   the global behavior of solutions to \eqref{PVI} with $\mathcal{ME}[u_0] < 1$ in the following theorem
\begin{teo}\label{dic_guz}
Let $u(x,t)$ be a solution of \eqref{PVI} and $0 < s_c < 1$. Assume $\mathcal{ME}[u_0] < 1$. Then
\begin{itemize}
        \item[(i)] If $\mathcal{MP}[u_0] > 1$, and either $V(0) < \infty,$ or $u_0$ is radial, or $N = 1$, then the solution blows up in finite time, in both time directions.
        \item[(ii)] If $\mathcal{MP}[u_0] < 1$, $N \geq 2$, and $u_0$ is radial then the solution is global and scatters, in both time directions.
\end{itemize}
\end{teo}
\begin{re}
 The case $\mathcal{MP}[u_0] = 1$ cannot occur if $\mathcal{ME}[u_0] < 1$ (see  \citet[Lemma 4.2, item (ii)]{FG_scat}.
\end{re}
\begin{re}In \citet{FG_scat} and \citet{Blowup_Dinh}, this theorem was proven using $\mathcal{MK}[u_0]$ instead of $\mathcal{MP}[u_0]$. We show the equivalence, if $\mathcal{ME}[u_0] \leq 1$ in Proposition \ref{MPequiv}. Therefore, as in the case $\mathcal{ME}[u_0] > 1$ the equivalence does not hold, the quantity that governs the dichotomy between blow-up and scattering is, in any case, $\mathcal{MP}[u_0]$.
\end{re}
We are interested here in criteria that includes initial data \textit{above} the threshold $\mathcal{ME}[u_0] = 1$.
The first theorem we prove is a dichotomy

\begin{teo}\label{teo1}
Let $u$ be a solution of \eqref{PVI}, where $1+\frac{2(2-b)}{N} < p < p^*_b$. Assume $N \geq 2$, $V(0)<\infty$, $u_0\in H^1(\Real^{N})$ and
\begin{equation}\label{cond1}
\mathcal{ME}[u_{0}]\left(1-\displaystyle\frac{(V_t(0))^2}{32E[u_0]V(0)}\right)\leq 1.
\end{equation}

\begin{itemize}
\item[(i)] (Blow-up) If
\begin{equation}\label{condb1}
\mathcal{MP}[u_0] > 1
\end{equation}
and
\begin{equation}\label{cond2}
V_t(0)\leq0,
\end{equation}

then $u(t)$ blows-up in finite positive, $T_{+}<\infty.$

\item[(ii)] (Boundedness and scattering) If
\begin{equation}\label{conds1}
\mathcal{MP}[u_0] < 1
\end{equation}
and
\begin{equation}\label{conds2}
V_t(0)\geq0,
\end{equation}
then
\begin{equation}\label{limsup}
\limsup_{t\to T_+(u)}M[u_0]^{1-s_c}\left(\int|x|^{-b}|u(t)|^{p+1}\right)^{s_c}<M[Q]^{1-s_c}\left(\int|x|^{-b}|Q|^{p+1}\right)^{s_c}.
\end{equation}
In particular, $T_+=+\infty$. Moreover, if $b < \min\left\{\frac{N}{3},1\right\}$ and $u$ is radial, then it scatters forward in time in $H^1$.
\end{itemize}
\end{teo}

\begin{re} If $\mathcal{ME}[u_0] < 1$, the conclusion of Theorem \ref{teo1} follows from Theorem \ref{dic_guz}. Theorem is new only in the case $\mathcal{ME} [u_0]\geq 1$.
\end{re}

\begin{re}
The proof of Theorem \ref{teo1} shows that there are two disjoint subsets (defined by \eqref{cond1}, \eqref{condb1} and \eqref{cond2}; and by \eqref{cond1}, \eqref{conds1} and \eqref{conds2}) that are stable under the INLS flow and contain solutions with arbitrary mass and energy (see, for example, Remark \ref{re_quad} below).
\end{re}

\begin{re}
We prove in Section \ref{s_scat} that any solution of \eqref{PVI} that satisfies \eqref{limsup} scatters for positive time. Replacing $\mathcal{MP}[u_0]$ by $\mathcal{MK}[u_0]$, this result is already known (see \citet{FG_scat}). Due to the one-sided implication \eqref{ida}, our assumption is weaker. Therefore, Theorem \ref{teo1} improves known results.
\end{re}

\begin{re}
The scattering statement of Theorem \ref{teo1} is optimal in the following sense: If $u_0 \in H^1(\Real^N)$ has finite variance and scatters forward in time, then there exists $t_0 \geq 0$ such that \eqref{cond1}, \eqref{conds1} and \eqref{conds2} are satisfied by $u(t)$ for all $t \geq t_0$. In fact, if $u(t)$ scatters forward in time, then $\displaystyle\int |x|^{-b}|u(t)|^{p+1}\to 0$. This implies $E[u_0] > 0$ and, by \eqref{d2var},
$$
V_t(t) \approx 16 E[u_0]t\quad\quad \text{and } V(t) \approx 8 E[u_0] t^2
$$

which implies $$\mathcal{ME}[u_{0}]\left(1-\displaystyle\frac{(V_t(t))^2}{32E[u_0]V(t)}\right)\to 0, \text{ as }t \to +\infty.$$
\end{re}

As a consequence of Theorem \ref{teo1}, we obtain
\begin{cor}\label{cor_gamma} Let $\gamma \in \Real\backslash\{0\}$, $v_0 \in H^1(\Real^N)$ with finite variance be such that $\mathcal{ME}[v_0] < 1$, and $u^\gamma$ be the solution of \eqref{PVI} with initial data
$$
u_0^\gamma = e^{i\gamma |x|^2}v_0.
$$

\begin{itemize}
\item[(i)] If $\mathcal{MP}[v_0] > 1$, then for any $\gamma <  0$, $u^\gamma$ blows up in finite positive time;
\item[(ii)] If $\mathcal{MP}[v_0] < 1$, then for any $\gamma > 0$, $u^\gamma$ satisfies \eqref{limsup}. Moreover, if $b < \min\left\{\frac{N}{3},1\right\}$ and $v_0$ is radial, then $u^\gamma$ scatters forward in time in $H^1(\Real^N)$.
\end{itemize}
\end{cor}

\begin{re}\label{re_quad}
With the above corollary, we can predict the behavior of a class of solutions with arbitrarily large energy. If $\mathcal{ME}[v_0] < 1$, then
$$
E[u_0^{\gamma}] = 4\gamma^2 \|xv_0\|_{L^2}^2+4\gamma \text{ \textup{Im}} \int x \cdot \nabla v_0 \bar{v}_0  + E[v_0]
$$

and $E[u_0^\gamma]  \to  +\infty$ as $\gamma \to \pm \infty$.
\end{re}

\begin{re} Note that the statement of Theorem \ref{teo1} is not symmetric in time as the statement of Theorem \ref{dic_guz}. Indeed, Corollary \ref{quad_Q} below shows solutions with different behaviors in positive and negative times.
\end{re}

\begin{cor}\label{quad_Q}Let $\gamma \in \Real$ and $Q^\gamma$ be the solution to \eqref{PVI} with initial data
$$
Q_0^\gamma = e^{i\gamma|x|^2}Q.
$$
\begin{itemize}
\item[(i)] If $\gamma > 0$, then $Q^\gamma$ is globally defined on $[0,+\infty)$, scatters forward in time and blows up backwards in time.
\item[(ii)] If $\gamma < 0$, then $Q^\gamma$ is globally defined on $(-\infty,0]$, scatters backward in time and blows up forward in time.
\end{itemize}
\end{cor}

\subsection{Blow-up criteria}

The blow up criterion of \citet{VPT}, Zakharov \cite{ZV72} and Glassey \cite{Glassey} for the NLS use the second derivative of the variance $V(t)$ to show that finite variance, negative energy solutions blow up in finite time. The second derivative of the variance is also used in \citet{lushnikov}, but with an approach based on classical mechanics, resulting in a finer blow-up criterion. This and and another criteria were proven in \citet{HPR} for the 3D cubic NLS. The argument was extended in \citet{DR_Going} to the focusing mass-supercritical NLS in any dimension, and examples were given to show that these new criteria are not equivalent to the previous ones. We extend these criteria for the focusing, mass-supercritical INLS equation in any dimension:

\begin{teo}\label{lush1} Suppose that $u_0\in H^1(\mathbb{R}^N)$, $N \geq 1$ and $V (0) < \infty$.  The following inequality is a sufficient condition for blow-up in finite time for solutions to (1.1) with $0 < s_c < 1$ and
$E[u_0] > 0$\begin{equation}
\frac{V_t(0)}{M[u_0]} < \sqrt{8Ns_c} g \left(\frac{4}{Ns_c}\frac{E[u_0]V(0)}{M[u_0]^2}\right),
\end{equation}
where
\begin{equation}\label{g}
g(x) =
\left\{
\begin{array}{l}
\sqrt{\frac{1}{kx^k}+x-(1+\frac{1}{k})}\, \mbox{ if }0 < x\leq 1\\
-\sqrt{\frac{1}{kx^k}+x-(1+\frac{1}{k})}\, \mbox{ if } x\geq 1
\end{array}
\right.
\mbox{ with }k = \frac{(p -1)s_c}{2}.
\end{equation}
\end{teo}
\begin{teo}\label{lush2} Suppose that $u_0\in H^1(\Real^N)$ and $V (0) < \infty$. The following inequality is a sufficient condition for blow-up in finite time for solutions to \eqref{PVI} with $0 < s_c < 1$ and $E[u_0]>0$
\begin{equation}
\frac{V_t(0)}{M[u_0]}<\frac{4\sqrt 2 M[u_0]^{\frac{1}{2}-\frac{p+1}{N(p-1)+2b}}E[u_0]^{\frac{s_c}{N}}}{C}g\left(C^2\frac{E[u_0]^{\frac{4}{N(p-1)+2b}}V(0)}{M[u_0]^{1+\frac{2(p+1)}{N(p-1)+2b}}}\right),
\end{equation}
where $g$ is defined in \eqref{g},
\begin{equation}
C=\left(\frac{2(p+1)}{s_c(p-1)}(C_{p,N})^{\frac{N(p-1)+2b}{2}+(p+1)}\right)^{\frac{2}{N(p-1)+2b}}.
\end{equation}
and $C_{p,N}$ the a sharp constant in the interpolation inequality \eqref{GN1}.
\end{teo}

\begin{re}
For real-valued initial data, Theorem \ref{lush2} is an improvement over Theorem \ref{lush1} if $$\mathcal{ME}[u_0] > \left(\frac{Ns_c C^2}{4}\right)^{\frac{N(p-1)+2b}{N(p-1)+2b-4}}.$$
\end{re}
\begin{re}
In both theorems, the restriction $s_c < 1$ is only needed to ensure the local well-posedness.
\end{re}

This paper is structured as follows: In section \ref{s_main_prop}, we prove the boundedness and blow-up part of Theorem \ref{teo1}. The scattering part is proven in section \ref{s_scat}. In section \ref{s_blowup_criteria}, we show two non-equivalent blow-up criteria for the INLS (Theorems \ref{lush1} and \ref{lush2}).

\section{Acknowledgments}
The authors thank Luiz Gustavo Farah (UFMG) and Svetlana Roudenko (FIU) for their valuable comments and suggestions which helped improve the manuscript. Part of this work was done when the first author was visiting the Institute for Pure and Applied Mathematics (IMPA), for which all authors are very grateful as it boosted the energy into the research project.

%% file: bound.tex
\section{Boundedness and Blow-up}\label{s_main_prop}

We start this section with the proof of the equivalence between using $\mathcal{MK}[u_0]$ and $\mathcal{MP}[u_0]$ in the dichotomy when $\mathcal{ME}[u_0] \leq 1$.

\begin{prop}\label{MPequiv}
If $f \in  H^1(\Real ^N)$, then
\begin{equation}\label{ida}
\mathcal{MK}[f] < 1 \Longrightarrow \mathcal{MP}[f]<1.
\end{equation}
Furthermore, assume $\mathcal{ME}[f] \leq 1$. Then
\begin{equation}\label{volta}
\mathcal{MK}[f] < 1 \Longleftrightarrow\mathcal{MP}[f]<1.
\end{equation}
\end{prop}

\begin{proof}
We write the sharp Gagliardo-Nirenberg inequality \eqref{GN1} as
\begin{equation}
\left(\mathcal{MP}[f]\right)^\frac{4}{N(p-1)+2b}\leq \mathcal{MK}[f],
\end{equation}
and \eqref{ida} follows. Now, if $\mathcal{MP}[f]< 1$ and $\mathcal{ME}[f] \leq 1$, then
\begin{equation}
M[Q]^\frac{1-s_c}{s_c}E[Q] \geq M[f]^\frac{1-s_c}{s_c}E[f] > \frac{1}{2} M[f]^\frac{1-s_c}{s_c}\int|\nabla f|^2\,dx-\frac{1}{p+1}M[Q]^\frac{1-s_c}{s_c}\int|x|^{-b}|Q|^{p+1}\,dx
\end{equation}
taking the first and last member, we conclude $\mathcal{MK}[f]< 1$.
\end{proof}
We also  point that the inequalities in \eqref{volta} can be replaced by equalities: we can scale $f$ so that $M[f] = M[Q]$. By similar arguments as the ones used in proving \eqref{ida} and \eqref{volta}, $\mathcal{MP}[f] = 1$ or $\mathcal{MK}[f] = 1$ in the case $\mathcal{ME}[f]\leq 1$, implies $\mathcal{MP}[f] = \mathcal{MK}[f] = \mathcal{ME}[f] = 1$. In this case, $f$ is equal to $Q$ up to scaling and phase.

We now turn to the proof of Theorem \ref{teo1}.  Start rewriting the Gagliardo-Nirenberg inequality \eqref{GN1} as
\begin{equation}\label{GN}
\left(\int|x|^{-b}|f|^{p+1}\,dx\right)^{\frac{4}{N(p-1)+2b}}\leq C_QM[f]^{\kappa}\int\, |\nabla u|^2\,dx\,\,\,\,\,\,\,\,\,\,\,\kappa=\frac{2(p+1)}{N(p-1)+2b}-1,
\end{equation}
where
\begin{align}
C_Q&:=(C_{p,N})^{\frac{4}{N(p-1)+2b}}=\frac{2(p+1)}{N(p-1)+2b}\frac{\left(\displaystyle\int|x|^{-b}|Q|^{p+1}\,dx\right)}{M[Q]^{\kappa}}^{\frac{4}{N(p-1)+2b}-1}\nonumber\\&
\nonumber\\&
=\left(\frac{8(p+1)}{A}\right)^{\frac{4}{N(p-1)+2b}}\frac{s_c(p-1)}{N(p-1)+2b}\cdot\frac{E[Q]}{M[Q]^{\kappa}}^{\frac{4}{N(p-1)+2b}-1}
\end{align}
and
\begin{equation}
    A:=2(N(p-1)+2b-4)=4(p-1)s_c.
\end{equation}

We use the following Cauchy-Schwarz inequality, proved by Banica \cite{Bn}. We include the proof here for the sake of completeness.

\begin{lemma}\label{cotaVt} Let $f\in H^1(\Real^N)$ such that $|x|f\in L^2(\Real^N)$. Then,
\begin{equation}
\left(\text{Im}\int x\cdot \nabla f\, \overline{f}\,dx\right)^2\leq \int|x|^2|f|^2\,dx\left[\int |\nabla f|^2\,dx-\frac{1}{C_QM^\kappa}\left(\int|x|^{-b}|f|^{p+1}\,dx\right)^{\frac{4}{N(p-1)+2b}}\right].
\end{equation}
\end{lemma}
\begin{proof}
Given $f \in H^1(\Real^N)$ and $\lambda>0$, we have
$$\nabla \left(e^{i\lambda|x|^2}f\right)=2i\lambda e^{i\lambda|x|^2}xf+e^{i\lambda|x|^2}\nabla f=e^{i\lambda|x|^2}(2i\lambda\,x f+\nabla f).$$
Thus,
\begin{align*}
\int \left|\nabla \left(e^{i\lambda|x|^2}f\right)\right|^2\,dx&=\int e^{i\lambda |x|^2}(2i\lambda\,x f+\nabla f)e^{-i\lambda |x|^2}(-2i\lambda\,x \overline{f}+\nabla \overline{f})\,dx\\
&=4\lambda^2\int|x|^2|f|^{2}\,dx+4\lambda \text{ Im}\int x\cdot \nabla f\,\overline{f}\,dx+\int|\nabla f|^2\,dx
\end{align*}
and from the Gagliardo-Niremberg inequality \eqref{GN}, for all $\lambda\in \Real$ we get
\begin{align}\label{pol1}
C_QM[f]^\kappa\left[4\lambda^2\int |x|^2|f|^2\,dx+4\lambda \text{ Im}\int x\cdot\nabla f\,\overline{f}\,dx+\int|\nabla f|^2\,dx\right]-\left(\int|x|^{-b}|f|^{p+1}\,dx\right)^{\frac{4}{N(p-1)+2b}}\geq 0.
\end{align}
Note that the left-hand side of inequality above is a quadratic polynomial in $\lambda$ . The discriminant of this polynomial is non-positive, wich yields the conclusion of the lemma.
\end{proof}
\begin{proof}[Proof of Theorem \ref{teo1}]
We will assume
\begin{align}\label{ME}
\mathcal{ME}[u_0]\geq 1,
\end{align}
as the case $\mathcal{ME}[u_0] < 1$ has been proven by \citet{FG_scat}.
By \eqref{d2var}, we have
\begin{equation}\label{enerviri2}
\int |\nabla u|^2\,dx=\frac{4(N(p-1)+2b)E[u_0]-V_{tt}}{A}.
\end{equation}

Furthermore,
\begin{align}
\int |x|^{-b}|u|^{p+1}\,dx&=(p+1)\frac{8\|\nabla u\|^2_2-V_{tt}}{4(N(p-1)+2b)}\nonumber\\&=(p+1)\frac{16E[u_0]-V_{tt}}{4(N(p-1)+2b)}+\frac{16}{4(N(p-1)+2b)}\int |x|^{-b}|u|^{p+1}\,dx.
\end{align}
Solving the equality above for $\displaystyle\int |x|^{-b}|u|^{p+1}\,dx$, we have
\begin{align}\label{enervirip}
\int |x|^{-b}|u|^{p+1}\,dx&=(p+1)\frac{16E[u_0]-V_{tt}}{2A}.
\end{align}
Note that the expression \eqref{enervirip} implies that $V_{tt}\leq 16E[u_0]$ for all $t$. In view of the equation     \eqref{dvar}, the derivative of variance $V(t)$, and Lemma \ref{cotaVt} we get,
\begin{align}\label{des1}
(V_t(t))^2&=16\left(\text{Im}\int x\cdot \nabla u(t)\,\overline{u}(t)\,dx\right)^2\nonumber\\&\leq 16\int V(t)\left[\int |\nabla u(t)|^2\,dx-\frac{1}{C_QM[u_0]^\kappa}\left(\int|x|^{-b}|u(t)|^{p+1}\,dx\right)^{\frac{4}{N(p-1)+2b}}\right].
\end{align}
If  $z(t)=\sqrt{V(t)}$, then
$$z_t(t)=\frac{1}{2}\frac{V_t(t)}{\sqrt{V(t)}}.$$
Dividing \eqref{des1} by $V(t)$, using \eqref{enerviri2}, \eqref{enervirip} and \eqref{des1}, we have
\begin{align*}
(z_t(t))^2&=\frac{1}{4}\frac{(V_t(t))^{2}}{V(t)}\nonumber\\&
\leq 4\left[\frac{4(N(p-1)+2b)E[u_0]-V_{tt}}{A}-\frac{1}{C_QM[u_0]^\kappa}\left(\frac{(p+1)(16E[u_0]-V_{tt})}{2A}\right)^{\frac{4}{N(p-1)+2b}}\right],
\end{align*}
that is,
\begin{align}\label{ztmefi}
(z_t(t))^2\leq 4\varphi(V_{tt}),
\end{align}
where
\begin{align}
\varphi(\alpha)=\left[\frac{4(N(p-1)+2b)E[u_0]-\alpha}{A}-\frac{1}{C_QM[u_0]^\kappa}\left(\frac{(p+1)(16E[u_0]-\alpha)}{2A}\right)^{\frac{4}{N(p-1)+2b}}\right]
\end{align}
is defined for $\alpha\in (-\infty,16E[u_0]]$. We have
\begin{align}
\varphi'(\alpha)=-\frac{1}{A}+\frac{4}{C_QM[u_0]^{\kappa}(N(p-1)+2b)}\left(\frac{p+1}{2A}\right)^{\frac{4}{N(p-1)+2b}}(16E[u_0]-\alpha)^{\frac{4}{N(p-1)+2b}-1}.
\end{align}
Consider $\alpha_m\in (-\infty, 16E[u_0])$ such that $\varphi'(\alpha_m)=0$, that is,
\begin{align}\label{alfa}
\frac{1}{A}=\frac{4}{C_QM[u_0]^{\kappa}(N(p-1)+2b)}\left(\frac{p+1}{2A}\right)^{\frac{4}{N(p-1)+2b}}(16E[u_0]-\alpha_m)^{\frac{4}{N(p-1)+2b}-1}.
\end{align}
Since $s_c>0$,
$$\frac{4}{N(p-1)+2b}-1=\frac{4-N(p-1)-2b}{N(p-1)+2b}=-\frac{2s_c}{(p-1)(N(p-1)+2b)}<0,$$
therefore $\varphi$ is decreasing on $(-\infty, \alpha_m)$ and increasing on $(\alpha_m,16E[u_0]]$. Note that \eqref{alfa} implies
\begin{align}
\frac{\alpha_m}{8}=\frac{(\alpha_m-16E)(N(p-1)+2b)}{4A}+\frac{4(N(p-1)+2b)E}{A}-\frac{\alpha_m}{A}=
\varphi(\alpha_m).
\end{align}
Using \eqref{alfa} and \eqref{cnstGN}, we have
\begin{align}
\frac{E[Q]}{M[Q]^{\kappa}}^{\frac{4}{N(p-1)+2b}-1}=\frac{\left(E[u_0]-\frac{\alpha_m}{16}\right)}{M[u_0]^\kappa}^{\frac{4}{N(p-1)+2b}-1},
\end{align}
hence raising both sides to $\frac{2(p-1)}{N(p-1)+2b}$, we get
\begin{align}\label{charac}
\left(\frac{M[u_0]}{M[Q]}\right)^{\frac{1-s_c}{s_c}}\frac{E[u_0]-\frac{\alpha_m}{16}}{E[Q]}=1.
\end{align}
As a consequence of \eqref{ME}
\begin{equation}
\left(\frac{M[u_0]}{M[Q]}\right)^{\frac{1-s_c}{s_c}}\frac{E[u_0]-\frac{\alpha_m}{16}}{E[Q]}=1 \leq \mathcal{ME}[u_0]=\left(\frac{M[u_0]}{M[Q]}\right)^{\frac{1-s_c}{s_c}}\frac{E[u_0]}{E[Q]},
\end{equation}
i.e.,
\begin{equation}
\alpha_m\geq0,
\end{equation}
and by \eqref{cond1} and  \eqref{charac},
\begin{align}\label{Malfa}
(z_t(0))^2&=-\left(1-\frac{(V_t(0))^2}{32E[u_0]V(0)}\right)\frac{8E[u_0]\mathcal{ME}[u_0]}{\mathcal{ME}[u_0]}+8E[u_0]\nonumber\\&\geq -\frac{8E[u_0]}{\mathcal{ME}[u_0]}\left(\frac{M[u_0]}{M[Q]}\right)^{\frac{1-s_c}{s_c}}\frac{E[u_0]-\frac{\alpha_m}{16}}{E[Q]}+8E[u_0]\nonumber\\&= \frac{\alpha_m}{2}=4\varphi(\alpha_m).
\end{align}
We first prove case (i) of Theorem \ref{teo1}. Suppose that $u\in H^{1}(\Real^N)$ satisfies \eqref{condb1} and \eqref{cond2}. Note that  \eqref{cond2} is equivalent to
\begin{equation}\label{ztmezero}
z_t(0)=\frac{V_t(0)}{2\sqrt{V(0)}}\leq 0.
\end{equation}
In view of \eqref{EQ}, the assumption \eqref{condb1} means
\begin{equation}
\left(\frac{M[u_0]}{M[Q]}\right)^{\frac{1-s_c}{s_c}}\frac{A\displaystyle\int |x|^{-b}|u_0|^{p+1}\,dx}{(p+1)E[Q]}=\left(\frac{M[u_0]}{M[Q]}\right)^{\frac{1-s_c}{s_c}}\frac{\displaystyle\int |x|^{-b}|u_0|^{p+1}\,dx}{\displaystyle\int |x|^{-b}|Q|^{p+1}\,dx}>1
\end{equation}
and consequently, from \eqref{enervirip}
\begin{align}\label{mealfa}
V_{tt}(0)=-\frac{2A}{p+1}\int |x|^{-b}|u_0|^{p+1}+16E[u_0]<\alpha_m.
\end{align}

Note that, for all $t>0$
\begin{align}\label{eqztt}
z_{tt}(t)=\frac{d}{dt}\left[\frac{V_t(t)}{2\sqrt{V(t)}}\right]=\frac{V_{tt}(t)}{2\sqrt{V(t)}}-\frac{(V_t(t))^2}{4\sqrt{V(t)^3}}=\frac{1}{z(t)}\left(\frac{V_{tt}(t)}{2}-(z_t(t))^2\right).
\end{align}
Hence from \eqref{Malfa} and \eqref{mealfa}, we have
$$z_{tt}(0)=\frac{1}{z(0)}\left(\frac{V_{tt}(0)}{2}-(z_t(0))^2\right)<\frac{1}{z(0)}\left(\frac{\alpha_m}{2}-\frac{\alpha_{m}}{2}\right)=0.$$
Suppose that $z_{tt}(\tilde{t})\geq 0$ for some $\tilde{t}$ belonging to $[0,T_+(u))$. Then, as  $z_{tt}$ is continuous on $[0,T_+(u))$, by the intermediate value theorem there exists $t_0\in (0,T_+(u))$ such that
$$\forall t \in[0,t_0),\,\,z_{tt}(0)<0\,\mbox{ and } z_{tt}(t_0)=0.$$
Thus for \eqref{Malfa} and \eqref{ztmezero}
\begin{align}
\forall t\in (0,t_0],\,\,z_{t}(t)<z_t(0)\leq -\sqrt{4\varphi(\alpha_m)}.
\end{align}
We have, thus,
\begin{align}
\forall t\in (0,t_0],\,\,z_{t}^{2}(t)> 4\varphi(\alpha_m).
\end{align}
Using the inequality above and \eqref{ztmefi},
\begin{align}
\forall t\in (0,t_0],\,\,4\varphi(V_{tt}(t))\geq z_{t}^{2}(t)>4\varphi(\alpha_m).
\end{align}
Therefore, $V_{tt}(t)\neq\alpha_m$ for $t\in (0,t_0]$. Since $V_{tt}(0)<\alpha_m$  and by the continuity of $V_{tt}$,
\begin{align}\label{Vttmealfa}
\forall t\in [0,t_0],\,\,V_{tt}(t)<\alpha_m.
\end{align}
Since $V_{tt}(t)\neq\alpha_m$ and by \eqref{Vttmealfa}, we get
$$z_{tt}(t_0)=\frac{1}{z(t_0)}\left(\frac{V_{tt}(t_0)}{2}-z_t^{2}(t_0)\right)<\frac{1}{z(t_0)}\left(\frac{\alpha_m}{2}-\frac{\alpha_m}{2}\right),$$
contradicting the definition of $t_0$. Therefore,
\begin{align}\label{zttmezero}
z_{tt}<0\mbox{ for all }t\in [0,T_+(u)).
\end{align}
By contradiction, suppose that $T_+(u)=+\infty$. From \eqref{ztmezero} and  \eqref{zttmezero},
$$\forall t>0,\,\, z_t(t)<z_t(0)\leq 0,$$
a contradiction with nonnegativity of $z(t)$.

We now prove case (ii) of Theorem \ref{teo1}. We assume, besides the conditions \eqref{cond1} and \eqref{ME}, that \eqref{conds1} and \eqref{conds2} hold. That implies, in the same way as we did in case (i),\begin{equation}\label{ztMzero}
z_t(0)\geq 0
\end{equation}
\begin{equation}\label{vttMalfa}
V_{tt}(0)>\alpha_m.
\end{equation}
We affirm that there is $t_0\geq 0$ such that
\begin{align}\label{ztt0}
z_t(t_0)>2\sqrt{\varphi(\alpha_m)}.
\end{align}
Indeed, by \eqref{Malfa} and \eqref{ztMzero},
\begin{equation}\label{ztM2fi}
z_t(0)\geq 2\sqrt{\varphi(\alpha_m)}.
\end{equation}
If $z_t(0)> 2\sqrt{\varphi(\alpha_m)}$, then choose $t_0=0$ and we have the result. If not,
$$z_{tt}(0)=\frac{1}{z(0)}\left(\frac{V_{tt}(0)}{2}-z_t^{2}(0)\right)>\frac{1}{z(0)}\left(\frac{\alpha_m}{2}-\frac{\alpha_m}{2}\right)=0,$$
by \eqref{vttMalfa} and \eqref{ztM2fi}. Hence, there is a small $t_0>0$  satisfying \eqref{ztt0}.

Let $\varepsilon_0$ be a positive small number and assume
\begin{equation}\label{ztesp01}
z_t(t_0)\geq 2\sqrt{\varphi(\alpha_m)}+2\varepsilon_0.
\end{equation}
We will show that, for all $t\leq t_0$
\begin{equation}\label{zteps0}
z_t(t)>2\sqrt{\varphi(\alpha_m)}+\varepsilon_0.
\end{equation}
Suppose \eqref{zteps0} is false, and define
\begin{equation}
t_1=\inf\{t\geq t_0;\, z_t(t)\leq 2\sqrt{\varphi(\alpha_m)}+\varepsilon_0\}.
\end{equation}
By \eqref{ztesp01} $t_1>t_0$. By continuity of $z_t$,
\begin{equation}\label{ztt1}
z_t(t_1)=2\sqrt{\varphi(\alpha_m)}+\varepsilon_0
\end{equation}
and
\begin{equation}\label{ztMeps}
\forall \in [t_0,t_1],\,\,\,\,z_t(t)\geq 2\sqrt{\varphi(\alpha_m)}+\varepsilon_0.
\end{equation}
In view of \eqref{ztmefi},
\begin{equation}\label{zt2}
\forall t\in [t_0,t_1],\,\, (2\sqrt{\varphi(\alpha_m)}+\varepsilon_0)^2\leq z_t^{2}(t)\leq 4\varphi(V_{tt}(t)).
\end{equation}
Hence, $\varphi(V_{tt}(t))>\varphi(\alpha_m)$ for all $t\in [t_0,t_1]$, so, $V_{tt}(t)\neq \alpha_m$ and by continuity $V_{tt}(t)>\alpha_m$ for $t\in[t_0,t_1]$.
Using the Taylor expansion of $\varphi$ around $\alpha=\alpha_m$, there exists $a>0$ such that, if $|\alpha-\alpha_m|\leq 1$, then
\begin{equation}\label{taylor}
\varphi(\alpha)\leq \varphi(\alpha_m)+a(\alpha-\alpha_m)^2.
\end{equation}
We show that there exists a universal constant  $D>0$ such that
\begin{equation}\label{vttepsi}
\forall\,t\in[t_0,t_1]\,\, V_{tt}(t)\geq \alpha_m +\frac{\sqrt{\varepsilon_0}}{D}.
\end{equation}
Consider two cases:
\item[a)] If $V_{tt}(t)\geq \alpha_m+1$, then for $D>0$ large,  we get \eqref{vttepsi}
\item[b)] If $\alpha_m<V_{tt}(t)\leq \alpha_m+1$, then by \eqref{zt2} and \eqref{taylor}, we obtain
$$(2\sqrt{\varphi(\alpha_m)}+\varepsilon_0)^2\leq (z_t(t))^2\leq 4\varphi(V_{tt}(t))\leq 4\varphi(\alpha_m)+4a(V_{tt}(t)-\alpha_m)^2.$$
Thus,
$$4\sqrt{\varphi(\alpha_m)}\varepsilon_0<4\sqrt{\varphi(\alpha_m)}\varepsilon_0+\varepsilon_0^2\leq 4a(V_{tt}-\alpha_m)^2,$$
and choosing $D=\sqrt a(\varphi(\alpha_m))^{-\frac{1}{4}}$, \eqref{vttepsi} holds.

Furthermore,  by \eqref{eqztt} and \eqref{ztMeps}
\begin{align*}
z_{tt}(t_1)&=\frac{1}{z(t_1)}\left(\frac{V_{tt}(t_1)}{2}-z_{t}^{2}(t_1)\right)\\
&\geq \frac{1}{z(t_1)}\left(\frac{\alpha_m}{2}+\frac{\sqrt \varepsilon_0}{2D}-(2\sqrt{\varphi(\alpha_m)}+\varepsilon_0)^2\right)\\
&\geq \frac{1}{z(t_1)}\left(\frac{\sqrt{\varepsilon_0}}{2D}-4\varepsilon\sqrt{\varphi(\alpha_m)}-\varepsilon_0^2\right)>0,
\end{align*}
if $\varepsilon_0$ is small enough. That is, $z_t$ is increasing close to $t_1$, contradicting \eqref{ztt1} and\eqref{ztMeps}. This shows \eqref{zteps0}.
Note that we have also shown that the inequality \eqref{vttepsi} holds for all $t\in [t_0,T_+(u))$. Hence, by \eqref{enervirip}, \eqref{EQ} and \eqref{charac}
\begin{align*}
M[u_0]^{1-s_c}\left(\int|x|^{-b}|u(t)|^{p+1}\, dx\right)^{s_c}
&=M[u_0]^{1-s_c}\left[\frac{p+1}{2A}(16E[u_0]-V_{tt}(t))\right]^{s_c}\\
&\leq M[u_0]^{1-s_c}\left[\frac{p+1}{2A}\left(16E[u_0]-\alpha_m-\frac{\sqrt{\varepsilon_0}}{D}\right)\right]^{s_c}\\
&<M[u_0]^{1-s_c}\left[\frac{p+1}{2A}\left(16E[u_0]-\alpha_m\right)\right]^{s_c}\\
&=M[u_0]^{1-s_c}\left[\frac{8(p+1)}{A}E[Q]\right]^{s_c}\\
&=M[Q]^{1-s_c}\left[\int|x|^{-b}|Q|^{p+1}\, dx\right]^{s_c}.
\end{align*}
\end{proof}

%% file: quadratic.tex
\subsection{Dichotomy for quadratic phase initial data}
We now prove Corollary \ref{cor_gamma}, except for the scattering statement, which will follow from the results in Section \ref{s_scat}.
\begin{proof}[Proof of Corollary \ref{cor_gamma}]
Let $v_0$ satisfy $\mathcal{ME}[v_0] < 1$, $\gamma \in \Real\backslash \{0\}$ and $u$ be the solution with initial data $u_0 = e^{i\gamma|x|^2}v_0$. We assume $$\mathcal{ME}[u_0] \geq 1$$ (otherwise the result follows from Theorem \ref{dic_guz}).

We will now show that $u_0$ satisfies the assumption of Theorem \ref{teo1}. We need to calculate
\begin{equation}\label{e_u0}
E[u_0] = E[v_0] + 2 \gamma \mbox{ Im}\int x \cdot \nabla v_0 \bar{v}_0\, dx + 2 \gamma^2 \int |x|^2|v_0|^2\,dx
\end{equation}
and
\begin{equation}
\mbox{ Im}\int \bar{u}_0\,x\cdot \nabla u_0\,dx = \mbox{ Im}\int\bar{v}_0\, x \cdot \nabla v_0\,dx + 2 \gamma \int|x|^2|v_0|^2\,dx.
\end{equation}
Rewriting the above equations,
\begin{equation}\label{EuEv}
E[u_0] - \frac{\left(\mbox{Im}\displaystyle\int \bar{u}_0\,x\cdot \nabla u_0\,dx\right)^2}{2\displaystyle\int |x|^2|u_0|^2\,dx} =
E[v_0] - \frac{\left(\mbox{ \\ Im}\displaystyle\int \bar{v}_0\,x\cdot \nabla v_0\,dx\right)^2}{2\displaystyle\int |x|^2|v_0|^2\,dx} \leq E[v_0],
\end{equation}
or,
\begin{equation}\label{me_u0}
 \mathcal{ME}[u_0]\left[1-\frac{\left(\mbox{Im}\displaystyle\int \bar{u}_0x\cdot \nabla u_0\right)^2}{2E[u_0]\int |x|^2|u_0|^2}\right] = \mathcal{ME}[v_0] \leq 1.
\end{equation}
Therefore, the assumption \eqref{cond1} follows from \eqref{dvar} and \eqref{me_u0}.

We will assume here $\gamma > 0$ and $\mathcal{MP}[v_0] < 1$, as the proof of the other case is very similar. First note that, since $\mathcal{ME}[v_0] < 1$ and $\int|x|^2|v_0|^2 > 0$, there is only one positive solution of
\begin{equation}\label{poly_v0}
M[v_0]^\frac{1-s_c}{s_c}\left(E[v_0] + 2 \gamma \mbox{ Im}\int x \cdot \nabla v_0 \bar{v}_0\,dx + 2 \gamma^2 \int |x|^2|v_0|^2\,dx\right)=M[Q]^\frac{1-s_c}{s_c} E[Q].
\end{equation}
Now, since $\mathcal{ME}[u_0] \geq 1$ and $\gamma >0$, \eqref{e_u0}, we have $\gamma \geq \gamma_c^+$, where $\gamma_c^+$ is the positive solution of \eqref{poly_v0}. Rewriting \eqref{poly_v0}, we have
\begin{equation}
 \gamma_c^+ \mbox{ Im}\int x \cdot \nabla v_0 \bar{v}_0 \, dx +  (\gamma_c^+)^2 \int |x|^2|v_0|^2\, dx=\frac{M[Q]^\frac{1-s_c}{s_c}E[Q]-M[v_0]^\frac{1-s_c}{s_c}E[v_0] }{2M[v_0]^\frac{1-s_c}{s_c}} > 0,
\end{equation}
which implies
\begin{equation}
\mbox{ Im}\int x \cdot \nabla v_0 \bar{v}_0 \, dx +  \gamma_c^+\int |x|^2|v_0|^2 \, dx> 0.
\end{equation}
Using that $\gamma \geq \gamma_c^+$, we see that
\begin{equation}
\mbox{ Im}\int x \cdot \nabla u_0 \bar{u}_0 \, dx = \mbox{ Im}\int x \cdot \nabla v_0 \bar{v}_0 \, dx +  \gamma \int |x|^2|v_0|^2 \, dx> 0,
\end{equation}
which yields \eqref{conds2}. Since Theorem \ref{teo1} applies, we conclude the proof.
\end{proof}

We next prove prove Corollary \ref{quad_Q}, except for the scattering statement.
\begin{proof}[Proof of Corollary \ref{quad_Q}]
Given that $\bar{u}(x,-t)$ is a solution of \eqref{PVI} if $u(x,t)$ is a solution, we can assume $\gamma > 0$. We only need to prove that $$\mbox{ Im}\displaystyle\int x\cdot \nabla Q^\gamma(t_0)\overline{Q^\gamma}(t_0)\,dx \geq 0,$$ $$\mathcal{MP}[Q^\gamma(t_0)] < 1$$ and
$$ME[Q^{\gamma}(t_0)]\left(1-\frac{(V_t(t_{0}))^2}{32E[Q^\gamma(t_0)]V(t_{0})}\right)\leq 1,$$ for some $t_0>0$, where $V(t) = \displaystyle\int|x|^2|Q^\gamma(x,t)|^2\,dx$. First note that, for $Q^\gamma_0=e^{i\gamma|x|^2}Q$, we have
\begin{equation}
\nabla Q^\gamma_0=(2i\gamma x Q+\nabla Q)e^{i\gamma|x|^2}, \text{ and}
\end{equation}
\begin{equation}\label{AQg}
\triangle Q^{\gamma}_0=e^{i\gamma|x|^2}(2iN\gamma Q+4i\gamma x\cdot\nabla Q-4\gamma^2|x|^2Q+\triangle Q).
\end{equation}
Thus,
\begin{align}\label{Vt>0}
\mbox{ Im}\int x\cdot \nabla Q_0^{\gamma}Q_0^\gamma\,dx&=\mbox{ Im}\int x\cdot (2i\gamma x Q+\nabla Q)e^{i\gamma|x|^2}e^{-i\gamma|x|^2}Q\, dx\\&=\mbox{ Im}\int x\cdot(2i\gamma x Q+\nabla Q)Q\,dx\nonumber\\
&=2\gamma\int|x|^2Q^2\,dx>0.
\end{align}
which shows $\mbox{Im}\displaystyle\int x\cdot \nabla Q^\gamma(t_0)\overline{Q^\gamma}(t_0)\,dx >0$ for sufficiently small $t_0$.
Moreover, using the fact that $Q^{\gamma}$ is a solution to $\eqref{PVI}$, we have
\begin{align}
\frac{d}{dt}\int|x|^{-b}|Q^{\gamma}|^{p+1}\,dx&=(p+1)\mbox{ Re}\int |x|^{-b}(\partial_tQ^{\gamma}\overline{Q^{\gamma}})|Q^{\gamma}|^{p-1}\,dx\nonumber\\
&=(p+1)\mbox{ Re}\int |x|^{-b}(i\triangle Q^{\gamma}\overline{Q^{\gamma}})|Q^{\gamma}|^{p-1}\,dx\nonumber\\
&=-(p+1)\mbox{ Im}\int |x|^{-b}|Q^{\gamma}|^{p-1}\triangle Q^\gamma \overline{Q^\gamma}\,dx.
\end{align}
Consequently, from \eqref{AQg},
\begin{align}
\left[\frac{d}{dt}\int|x|^{-b}|Q^{\gamma}|^{p+1}\,dx\right]\Bigg|_{t=0}&=\left[-(p+1)\mbox{ Im}\int |x|^{-b}|Q^{\gamma}|^{p-1}\triangle Q^\gamma \overline{Q^\gamma}\,dx\right]\Bigg|_{t=0}\nonumber\\
&=-(p+1)\mbox{ Im}\int |x|^{-b}|Q^{\gamma}_0|^{p-1}Q(2iN\gamma Q+4i\gamma x\cdot \nabla Q\\&\quad\quad\quad\quad\quad\quad\quad\quad\quad\quad\quad\quad\quad\,\,\,-4\gamma^2|x|^2Q+\triangle Q)\,dx\nonumber\\
&=-2N\gamma(p+1)\int |x|^{-b}Q^{p+1}\,dx-4\gamma(p+1)\int Q^px\cdot \nabla Q\,dx\nonumber\\
&=-2N\gamma(p-1)\int |x|^{-b}Q^{p+1}\,dx<0.
\end{align}
Since
\begin{equation}
M[Q^{\gamma}_0]^{\frac{1-s_c}{s_c}}\int |x|^{-b}|Q^{\gamma}_0|^{p+1}\,dx=M[Q]^{\frac{1-s_c}{s_c}}\int |x|^{-b}|Q|^{p+1}\,dx,
\end{equation}
we get, for sufficiently small $t_0$
\begin{equation}
\mathcal{MP}[Q^{\gamma}(t_0)]<1.
\end{equation}
Now, define the function $F$ as
\begin{equation}\label{F}
F(t)=M[Q^{\gamma}]^{\frac{1-s_c}{s_c}}\left[E[Q^\gamma]-\frac{\left(\text{ Im}\displaystyle\int x\cdot \nabla Q^\gamma(t)\overline{Q^{\gamma}}(t)\,dx\right)^2}{2\displaystyle\int |x|^2|Q^{\gamma}(t)|^2\,dx}\right]-M[Q]^{\frac{1-s_c}{s_c}}E[Q].
\end{equation}
In view of \eqref{EuEv}, with $v_0=Q$, we conclude $F(0)=0$. We just need to check that $F(t)\leq 0$ for small positive $t.$
Let$$V(t)=\int |x|^2|Q^{\gamma}(x,t)|^2\,dx,\,\,\,\,\,\,\quad z(t)=\sqrt{V(t)}.$$
We can rewrite \eqref{F} as
$$
F(t)=M[Q^{\gamma}]^{\frac{1-s_c}{s_c}}\left(E[Q^\gamma]-\frac{1}{8}(z_t(t)^2)\right)-M[Q]^{\frac{1-s_c}{s_c}}E[Q],
$$
and thus, 
$$
F_t(t)=-\frac{1}{4}M[Q^{\gamma}]^{\frac{1-s_c}{s_c}}z_t(t)z_{tt}(t).
$$
Using \eqref{dvar}, \eqref{d2var} and the fact that Gagliardo-Nirenberg inequality \eqref{GN1} is an equality for $f = Q = e^{-i\gamma|x|^2}Q^\gamma_0$, we conclude that $z_{tt}(0) = 0$. Therefore, 
\begin{align}
F_{tt}(0)&=-\frac{1}{4}M[Q^{\gamma}]^{\frac{1-s_c}{s_c}}\left(z_t(0)z_{ttt}(0)+(z_{tt}(0))^2\right),\\
&=-\frac{1}{4}M[Q^{\gamma}]^{\frac{1-s_c}{s_c}}z_t(0)z_{ttt}(0).
\end{align}
On the other hand,
\begin{equation}
V_{tt}=2(z_t)^2+2zz_{tt},\quad\quad V_{ttt}=6z_tz_{tt}+2zz_{ttt}.
\end{equation}
Thus, $V_{ttt}(0)=2z(0)z_{ttt}(0)$. Hence, $F_{tt}(0)$ and $-V_{ttt}(0)$ have the same sign, but from \eqref{Vt>0} $z_t(0)>0$. By \eqref{enervirip}, we get that this sign is the same as the one of
$$\left[\frac{d}{dt}\int |x|^{-b}|Q^\gamma|^{p+1}\,dx\right]\Bigg|_{t=0}= -\frac{(p+1)}{2A} V_{ttt}(0).$$
Therefore, $F_{tt}(0)<0$, which shows that $F(t)$ is negative for small $t> 0$. This completes the proof.
\end{proof}

%% file: scattering.tex

%
%
\section{Scattering}\label{s_scat}
We now prove the scattering part of theorem \ref{teo1}. We start with a lemma:

\begin{lemma}\label{energylemma}
Let $0 < a < A < \left(\displaystyle\int |x|^{-b}|Q|^{p+1}\right)^{s_c}M[Q]^{1-s_c}$. Then, there exists $\epsilon_0 = \epsilon_0(a,A)$ such that for all $f \in H^1(\mathbb{R}^N)$ with
$$
a \leq \left(\int |x|^{-b}|f|^{p+1}\,dx\right)^{s_c}M[f]^{1-s_c} \leq A,
$$
one has
\begin{equation}\label{el1}
\int \left|\nabla f\right|^2\,dx-\frac{N(p-1)+2b}{2(p+1)}\int|x|^{-b}|f|^{p+1}\,dx \geq \epsilon_0 M[f]^{1-\frac{1}{s_c}}
\end{equation}
and
\begin{equation}\label{el2}
E[f] \geq \frac{\epsilon_0}{2}M[f]^{1-\frac{1}{s_c}}.
\end{equation}
\end{lemma}

\begin{proof}
Recalling the sharp Gagliardo-Nirenberg inequality, we have:
\begin{align}
&M[f]^{\frac{1}{s_c}-1}\left[\int|\nabla f|^2\,dx-\frac{N(p-1)+2b}{2(p+1)}\int |x|^{-b}|f|^{p+1}\,dx\right]\nonumber\\
&\geq \frac{1}{c_Q}M[f]^{\frac{1}{s_c}-1-\kappa}\left(\int|x|^{-b}|f|^{p+1}\,dx\right)^{\frac{4}{N(p-1)+2b}} - M[f]^{\frac{1}{s_c}-1}\frac{N(p-1)+2b}{2(p+1)}\int |x|^{-b}|f|^{p+1}\,dx\nonumber\\
&= \frac{y^{\frac{4}{N(p-1)+2b}}}{c_Q}-\frac{N(p-1)+2b}{2(p+1)}y.
\end{align}
where $y=M[f]^{\frac{1}{s_c}-1}\displaystyle\int|x|^{-b}|f|^{p+1}\,dx$. The function $y \mapsto\frac{y^{\frac{4}{N(p-1)+2b}}}{c_Q}-\frac{N(p-1)+2b}{2(p+1)}y$
has only one zero $y^*$ on $(0,+\infty)$ and is positive on $(0,y^*)$. Since the inequality (\ref{energylemma}) is an equality when $f = Q$, $y^*$  is exactly $M[Q]^{\frac{1}{s_c}-1}\displaystyle\int|x|^{-b}|Q|^{p+1}\,dx$, and (\ref{el1}) follows. Noting that
$$
E[f] \geq \frac{1}{2}\left(\int|\nabla f|^2\,dx-\frac{N(p-1)+2b}{2(p+1)}\int|x|^{-b}|f|^{p+1}\,dx\right),
$$
we get (\ref{el2}), because $\frac{N(p-1)+2b}{4} \geq 1$.
\end{proof}
\begin{definicao}\label{Hs_adm} If $N \geq 1$ and $s \in (0,1)$, the pair $(q,r)$ is called $\dot{H}^s$\textit{-admissible} if it satisfies the condition
\begin{equation}\label{hs_adm_eq}
\frac{2}{q} = \frac{N}{2}-\frac{N}{r}-s,    
\end{equation}
where
$$
2 \leq q,r \leq \infty, \text{ and } (q,r,N) \neq (2,\infty,2).
$$
\end{definicao}
Also, considering the following closed subset of $H^s$-admissible pairs 
\begin{equation}
    \mathcal{A}_s = \left\{(q,r)\text{ is } \dot{H}^{s}\text{-admissible} \left|\,
    \begin{cases} 
    \left(\frac{2N}{N-2s}\right)^+ \leq r \leq \left(\frac{2N}{N-2}\right)^-
    , & N \geq 3\\
    \left(\frac{2}{1-s}\right)^+ \leq r \leq \left(\left(\frac{2}{1-s}\right)^+\right)', & N = 2\\ 
    \frac{2}{1-2s} \leq r \leq \infty, &N = 1
    \end{cases}   
    \right.
    \right\}
\end{equation}

where $a^+ = a + \epsilon$, for a fixed, small $\epsilon >0$ and $(a^+)'$ is defined as the number such that
$$
\frac{1}{a} = \frac{1}{a^+} + \frac{1}{(a^+)'},
$$
we define the scattering norm
\begin{equation}
    \|u\|_{S(\dot{H}^{s_c})} = \sup_{(q,r) \in \mathcal{A}_{s_c}}\|u\|_{L^q_t L^r_x}.
\end{equation}

It is already known that scattering follows from the uniform boundedness of the $H^1$ norm and the finiteness of the $S(\dot{H}^{s_c})$ norm (see \citet[Proposition 1.4]{FG_scat}).

\begin{prop}
Define $S(L,A)$ as the supremum of $\|u\|_{S(\dot{H}^{s_c})}$ such that $u$  is a radial solution to (\ref{PVI}) on $[0, +\infty)$ with
\begin{equation}\label{meL}
\mathcal{ME}[u_0] \leq L
\end{equation}
and \begin{equation}\label{mpA}
\sup_{t \in[0,+\infty)} \left(\int |x|^{-b}|u(t)|^{p+1}\,dx\right)^{s_c}M[u]^{1-s_c} \leq A.
\end{equation}

If $A < \left(\displaystyle\int|x|^{-b}Q^{p+1}\,dx\right)^{s_c}M[Q]^{1-s_c}$, then $S(L,A) < +\infty$.
\end{prop}

\begin{proof}
The proof goes along the spirit of \citet{DR_Going}, \citet{FG_scat} and (see also \citet{Guevara}). We will give an outline of the proof, highlighting the main differences.

First we note that, if $L > 0$ is small enough (i.e., $L^{s_c} < E[Q]^{s_c}M[Q]^{1-s_c})$, then $S(L,A) < +\infty$. Assume, by contradiction, that $S(L,A) = +\infty $ for some $L \in \mathbb{R}$. Note that, if $u \not\equiv 0$ satisfies \eqref{mpA}, with $ A < \left(\displaystyle\int|x|^{-b}Q^{p+1}\,dx\right)^{s_c}M[Q]^{1-s_c}$, then by Lemma \ref{energylemma}, $E[u] >0$. Thus, the quantity $L_c$ given by
$$
L_c = L_c(A) := \inf\left\{L \in \mathbb{R} \text{ s.t. } S(L,A) = +\infty\right\}
$$
is well-defined and positive.

Moreover, there exists a sequence $\{u_n\}$ of (global) radial solutions such that $$M[u_n] = 1,$$ $$\left\|u_n\right\|_{S\left(\dot{H}^{s_c}\right)} \rightarrow +\infty,$$
$$
E[u_n] \searrow L_c,
$$
and
$$
\sup_{t \in [0, +\infty)}\int|x|^{-b}|u|^{p+1} \, dx \leq A.
$$

Therefore, using the radial linear profile decomposition (\citet[Proposition 5.1]{FG_scat})  for the initial conditions $u_{n,0}$ (note that $\{u_{n,0}\}$ is bounded in $H^1(\Real^N)$) and the existence of wave operators for large times (see \citet{FG_scat} and \citet{Guevara}), we obtain, for each $M \in \mathbb{N} $ (passing, if necessary, to a subsequence) a nonlinear profile decomposition of the form:
\begin{equation}\label{nlpd}
u_{n,0} = \sum_{j=1}^M \tilde{u}^j\left(-t_n^j\right)+\tilde{W}_n^M,
\end{equation}

where, for each $j$, $\tilde{u}^j$ is a solution to (\ref{PVI}) and:

\begin{enumerate}
\item for $k \neq j$, $|t_n^k - t_n^j| \rightarrow +\infty$;
\item for each $j$, there exists $T_j > 0$ such that, if $t_n^j \rightarrow +\infty$, then $\tilde{u}^j$ is defined on $(-\infty,-T_j]$, and if $t_n^j \rightarrow -\infty$, then $\tilde{u}^j$ is defined on $[T_j,+\infty)$;
\item for each $j$, there exists $v^j \in H^1$ such that $\|\tilde{u}^j\left(-t_n^j\right) - e^{-it_n^j \Delta}v^j\|_{H^1} \rightarrow 0$;
\item $\displaystyle\lim_{M \rightarrow +\infty} \left[\lim_{n\rightarrow +\infty}\left\|e^{it\Delta}\tilde{W}_n^M\right\|_{S\left(\dot{H}^{s_c}\right)}\right] = 0$;
\item for fixed $M \in \mathbb{N}$ and any $0\leq s \leq 1$, the asymptotic Pythagorean expansion:
$$
\left\|u_{n,0}\right\|^2_{\dot{H}^s} = \sum_{j = 1}^M \left\|\tilde{u}^j\left(-t_n^j\right)\right\|^2_{\dot{H}^s}+\left\|\tilde{W}_n^j\right\|^2_{\dot{H}^s}+o_n(1)
$$
and the energy Pythagorean decomposition:
$$
E[u_{n,0}] = \sum_{j = 1}^M E\left[\tilde{u}^j\right]+E\left[\tilde{W}_n^j\right]+o_n(1).
$$
\end{enumerate}
We denote the solution to \eqref{PVI} in time $t$, with initial data $\psi$ by $\text{INLS}(t)\psi$. Note that, unlike in \citet{FG_scat}, we do not know whether the nonlinear profiles evolve into global solutions, because the quantity $E[\tilde{u}^j]^{s_c}M[\tilde{u}^j]^{1-s_c}$ may not be small. Thus, in order to prove that $\text{INLS}(t)\tilde{u}^j(-t_n^j)$ exists on $[0, +\infty)$, we need to track $\left\|\nabla \text{INLS}(t)\tilde{u}^j(-t_n^j)\right\|_{L^2}$.

Using long-time perturbation theory (\citet[Proposition 4.14]{FG_scat}), the asymptotic orthogonality at $t = 0$  can be extended to the \textit{INLS} flow.

\begin{lemma}(Pythagorean decomposition along the bounded INLS
flow).
Suppose $u_{n,0}$ is a radial bounded sequence in $H^1(\Real^N)$. Let $T \in (0, +\infty)$  be a fixed time. Assume that $u_n(t) = \mbox{INLS}(t)u_{n,0}$ exists up to time $T$ for all $n$; and $\displaystyle\lim_n \left\|\nabla u_n(t)\right\|_{L ^\infty_{[0,T]}L^2_x} < +\infty$.  Consider the nonlinear profile decomposition (\ref{nlpd}) and denote $W_n^M(t) = \text{INLS}(t)W^m_n$. Then for all $j$, the nonlinear
profiles $\tilde{v}^j(t) =  \text{INLS}(t)\tilde{u}^j(-t_n^j)$ exist up to time T and for all $t \in [0,T]$,

$$
\left\|\nabla u_{n}(t)\right\|^2_{L^2} = \displaystyle\sum_{j = 1}^M \left\|\nabla \tilde{v}^j\left(t\right)\right\|^2_{L^2}+\left\|\tilde{W}_n^j(t)\right\|^2_{L^2}+o_n(1)
,$$

where $o_n(1) \rightarrow 0$ uniformly on $0 \leq t \leq T$.
\end{lemma}
Invoking (\ref{meL}) and (\ref{mpA}) and using this orthogonality along the INLS flow, one is able to prove that $v^j(t)$ is defined on $[0, +\infty)$ as well, and satisfies, for every $j$,
$$
M[v^j] \leq 1,
$$
\begin{equation}
\mathcal{ME}[v^j] \leq L_c
\end{equation}
and \begin{equation}
\sup_{t \in[0,+\infty)} \left(\int |x|^{-b}|v^j(t)|^{p+1}\,dx\right)^{s_c}M[v^j]^{1-s_c} \leq A.
\end{equation}

The rest of the proof follows the same lines as \citet{DR_Going} and \citet{FG_scat}, using the criticality of $L_c$ to show the existence of only one non-zero profile, say, $v^1(t)$, and letting $u_c(t) = v^1(t)$. This criticality also shows that $M[u_c] = 1$ and $\mathcal{ME}[u_c] = L_c$. Long-time perturbation theory yields $\left\|u_c\right\|_{S\left(\dot{H}^{s_c}\right)}= +\infty$. At this point, the classical compactness lemma follows.
\begin{lemma}[Compactness]\label{compactness}
Assume that there exists $L_0 \in \mathbb{R}$ and a positive number  
$$A < \left(\displaystyle\int|x|^{-b}|Q|^{p+1}\,dx\right)^{s_c}M[Q]^{1-s_c}$$ such that $S(L_0,A) = +\infty$. Then there exists a radial global  solution $u_c$ of (\ref{PVI}) such that the set
$$
K = \displaystyle\left\{u_c(x,t), t \in [0, +\infty)\right\}
$$
has a compact closure in $H^1(\mathbb{R}^N)$.
\end{lemma}

Using this compactness lemma and the virial identity \eqref{d2var}, we also have the classic rigidity lemma.
\begin{lemma}[Rigidity]
There's no solution $u_c$ of (\ref{PVI}) satisfying the conclusion of Lemma \ref{compactness}.
\end{lemma}

The proof goes on the same lines as in \citet{DR_Going} and \citet{FG_scat}. We point here that the restriction $b < \min\left\{\frac{N}{3},1\right\}$ is technical and comes from the proof of long-time perturbation in \citet{FG_scat}.
\end{proof}

%% file: blow.tex
\section{Proof of the blowup criteria}\label{s_blowup_criteria}

In this section we prove two criteria for blow up in finite time. The first one is a generalization
of Lushnikov's criterion in \cite{lushnikov} and of Holmer-Platte-Roudenko criteria in \cite{HPR} for the INLS, and the second one is the modification of the first approach, where the
generalized uncertainty principle is replaced by the interpolation inequality \eqref{interp}. The two criteria are the INLS versions of the criteria proved by Duyckaerts and Roudenko in \cite{DR_Going}.

\begin{proof}[Proof of Theorem \ref{lush1}]
Integrating by parts,
\begin{align*}
\|u\|^{2}_{L^2}&=\int |u|^2\,dx=\frac{1}{N}\sum^{N}_{j=1}\int \partial_jx_j|u|^2\,dx=-\frac{1}{N}\sum_{j=1}^{N}\int x_j\partial_j(|u|^2)\,dx\\&=-\frac{1}{N}\sum_{j=1}^{N}\int x_j(\partial_ju\overline{u}+u\partial_j\overline{u})\,dx=-\frac{2}{N}\sum_{j=1}^{N}\text{Re}\,\int x_j\partial_ju\overline{u}\,dx\\&=-\frac{2}{N}\text{Re}\,\int (x\cdot\nabla u)\overline{u}\,dx.
\end{align*}
Since $|z|^2=|\text{Re}\,z|^2+|\text{Im}\,z|^2$, using Hölder's inequality
\begin{align}
\|xu\|^{2}_{L_2}\|\nabla u\|^{2}_{L^2}&\geq \left|\int(x\cdot\nabla u)\overline{u}\,dx\right|^2=\left|\text{Re}\,\int(x\cdot\nabla u)\overline{u}\,dx\right|^2+\left|\text{Im}\,\int(x\cdot\nabla u)\overline{u}\,dx\right|^2\nonumber\\&=\frac{N^2}{4}\|u\|^4_{L^2}+\left|\text{Im}\,\int(x\cdot\nabla u)\overline{u}\,dx\right|^2.
\end{align}
From the definition of variance and the identity for the first derivative of the variance \eqref{dvar}, we get the uncertainty principle
\begin{align}\label{prinuncer}
\frac{N^2}{4}\|u_0\|^2_{L^2}+\left|\frac{V_t}{4}\right|^2\leq V(t)\|\nabla u(t)\|^{2}_{L^2}.
\end{align}
Using the equation \eqref{d2var} for the second derivative of the variance, we obtain
\begin{align}\label{43}
V_{tt}(t)=4(N(p-1)+2b)E[u_0]-4(p-1)s_c\|\nabla u(t)\|^{2}_{L^2}.
\end{align}
 Substituting \eqref{43} in the uncertainty principle \eqref{prinuncer}, we have
\begin{align}\label{prinineq}
V_{tt}(t)\leq4(N(p-1)+2b)E[u_0]-N^2(p-1)s_c\frac{(M[u_0])^2}{V(t)}-\frac{(p-1)s_c}{4}\frac{|V_t(t)|^2}{V(t)}.
\end{align}
Now, we rewrite equation \eqref{prinineq} in order to cancel the term $V_t^2$. For this, define\begin{equation}\label{B}
V=B^{\frac{1}{\alpha+1}},\,\,\,\,\,\,\,\,\,\,\,\,\,\,\,\,\,\,\,\alpha=\frac{(p-1)s_c}{4}=\frac{N(p-1)-4+2b}{8}.
\end{equation}
Then,
\begin{align*}
V_t=\frac{1}{\alpha+1}B^{-\frac{\alpha}{\alpha+1}}\,\,\,\,\,\,\,\,\mbox{ and }\,\,\,\,\,\,\,\,V_{tt}=-\frac{\alpha}{(\alpha+1)^2}B^{-\frac{2\alpha+1}{\alpha+1}}B_t^2+\frac{1}{\alpha+1}B^{-\frac{\alpha}{\alpha+1}}B_{tt},
\end{align*}
which gives
$$
B_{tt}\leq 4(\alpha+1)N(p-1)E[u_0]B^{\frac{\alpha}{\alpha+1}}-(\alpha+1)N^2(p-1)s_c(M[u_0])^2B^{\frac{\alpha-1}{\alpha+1}}
$$
that is, for all $t\in [0,T_+(u)$
\begin{align}
B_{tt}\leq \frac{N(p-1)(N(p-1)+4+2b)}{2}\left(E[u_0]B^{\frac{N(p-1)-4+2b}{N(p-1)+4+2b}}-\frac{Ns_c}{4}(M[u_0])^{2}B^{\frac{N(p-1)-12+2b}{N(p-1)+4+2b}}\right).
\end{align}
In order to further simplify inequality, let us make a rescaling.
Define $B(t)=\mu \Phi(\lambda t)$, with
\begin{equation}\label{lamandmu}
\mu=\left(\frac{Ns_c(M[u_0])^2}{4E[u_0]}\right)^{\frac{N(p-1)+4+2b}{8}},\,\,\,\,\,\,\,\,\,\lambda=\frac{8\sqrt 2}{\sqrt{Ns_c}}\frac{E[u_0]}{M[u_0]}.
\end{equation}
Then letting $s=\lambda t$, we obtain
\begin{align}\label{eqmech}
\omega \Phi_{ss}\leq \Phi^{\gamma}-\Phi^{\delta}, \,\,\,s\in[0,T_+/a),
\end{align}
where
$$\gamma=\frac{N(p-1)-4+2b}{N(p-1)+4+2b},\,\,\,\,\,\,\,\,\,\,\,\,\,\,\,\,\,\delta=\frac{N(p-1)-12+2b}{N(p-1)+4+2b}=2\gamma-1,$$
$$ \omega=\frac{64}{N(p-1)(N(p-1)+4+2b)}$$
and since $p>1+\frac{4}{N}$,
\begin{align}
0<\gamma<1,\,\,\,\,\,-1<\delta<\gamma.
\end{align}
We rewrite \eqref{eqmech} as
\begin{equation}\label{eqmech2}
\omega \Phi_{ss}+\frac{\partial U}{\partial \Phi}\leq 0,
\end{equation}
for $t\in [0,T_+/a)$, where $U(\Phi)=\frac{\Phi^{\delta+1}}{\delta+1}-\frac{\Phi^{\gamma+1}}{\gamma+1}$. 
Define the energy of the particle
\begin{equation}
\mathcal{E}(s)=\frac{\omega}{2}\Phi_s^2(s)+U(\Phi(s))
\end{equation}
which is conserved for solutions of
\begin{equation}
\omega \Phi_{ss}+\frac{\partial U}{\partial \Phi}=0.
\end{equation}
Based on the ideas of Lushnikov \cite{lushnikov}, Duyckaerts and Roudenko \cite{DR_Going} studied this model and showed the following proposition
\begin{prop}\label{proplush}
Let $\Phi$ be a nonnegative solution of \eqref{eqmech2} such that one of the following holds:
\item[(\textbf{A})] $\mathcal{E}(0)<U_{max}\mbox{ and }\Phi(0)<1,$
\item[(\textbf{B})] $\mathcal{E}(0)>U_{max}\mbox{ and }\Phi_s(0)<0,$
\item[(\textbf{C})] $\mathcal{E}(0)=U_{max}, \Phi_s(0)<0\mbox{ and }\Phi(0)<1.$\\

Then $T_+<\infty$.
\end{prop}
\begin{proof}
For the sake of completeness of this work, we will give the proof of the proposition. Multiplying equation \eqref{eqmech2} by $\Phi_s$, we get
\begin{equation}\label{energszero}
\Phi_s(s)>0\Rightarrow \mathcal{E}_s(s)<0,\,\,\,\,\,\,\Phi_s(s)<0\Rightarrow \mathcal{E}_s(s)>0.
\end{equation}
We argue by contradiction, assuming $T_+=T_+(u)=+\infty$.

We first assume $(A)$. Let us prove by contradiction that\begin{equation}
\exists\,s>0,\,\,\,\Phi_s(s)<0.
\end{equation}
If not, $\Phi_s(s)\leq 0$ for all $s,$ and \eqref{energszero} implies that the energy decays. By $(A)$, $\mathcal{E}(s)\leq \mathcal{E}(0)<U_{max}$ for all $s$. Thus, $|\Phi(s)-1|\geq \varepsilon_0$ (where $\varepsilon_0>0$ depends on $\mathcal{E}(0))$ for all $s$. Since by $(A)$ $\Phi(0)<1$, we obtain by continuity of $\Phi$ that $\Phi(s)\leq 1-\varepsilon_0$ for all $s$. By equation \eqref{eqmech}, we deduce $\Phi_{ss}\leq -\varepsilon_1$ for all s, where $\varepsilon_1>0$ depends on $\varepsilon_0.$ Thus, $\Phi$ is strictly concave, a contradiction with the fact that $\Phi$ is positive and $T_+=+\infty.$ 

We have proved that there exists $s> 0$ such that $\Phi_s(s)< 0$. Letting
$$t_1=\inf \{s> 0; \Phi_s(s)< 0\},$$
we get by \eqref{energszero} that the energy is nonincreasing on $[0,t_1]$. Thus, $\mathcal{E}(s)<\mathcal{E}(0)\leq U_{max}$ on $[0,t_1]$, which proves that $\Phi(s)\neq 1$ on $[0,t_1]$. Since $\Phi(0)<1,$ we deduce by the intermediate value theorem that $\Phi(t_1)<1$ and by \eqref{eqmech} that $\Phi_{ss}(t_1)<0$. Since $\Phi_s(t_1)\leq 0$, an elementary bootstrap argument, together with equation \eqref{eqmech} shows that $\Phi(s)\leq 1-\varepsilon_0,\,\Phi_s(s)<0$ and $\Phi_{ss}(s)\leq-\varepsilon_1$ for $s>t_1$, for some positive constants $\varepsilon_0, \varepsilon_1$. This is again a contradiction with the positivity of $\Phi$.

We next assume (B). Let $t_1$ be such that $\Phi_s(s)<0$ on $[0,t_1]$. By \eqref{energszero}, $\mathcal{E}$ is nondecreasing on $[0,t_1]$, and thus, $\mathcal{E}(s)\geq \mathcal{E}(0)>U_{max}$ for all $s$ on $[0,t_1]$. As a consequence, $\frac{1}{2}\Phi_s(s)^2\geq \mathcal{E}(0)-U_{max}>0$ for all $s$ in  $[0,t_1]$, which shows that the inequality $\Phi_s(s)\leq -\sqrt{\mathcal{E}(0)-U_{max}}$ holds on $[0,t_1]$. Finally, an elementary bootstrap argument shows that the inequality $\Phi_s(s)\leq -\sqrt{\mathcal{E}(0)-U_{max}}$ is valid for all $s\geq 0$, a contradiction with the positivity of $\Phi$.

Finally, we assume (C). By bootstrap again, $\Phi_s(s)<0$, $\Phi(s)<1$ and $\Phi_{ss}(s)<0$ for all positive $s$, proving again that $\Phi$ is a strictly concave function, a contradiction.
\end{proof}

Since
\begin{equation}
\alpha=\frac{(p-1)s_c}{4}=\frac{N(p-1)-4+2b}{8}, 
\end{equation}  
we have
\begin{align}
2\alpha+1=\frac{N(p-1)+2b}{4}, \quad
\alpha+1=\frac{N(p-1)+4+2b}{8},
\end{align}
\begin{align} 
(\alpha+1)(\delta+1)=2\alpha,\quad
(\alpha+1)(\gamma+1)=2\alpha+1 \text{ and }
\omega=\frac{2}{(2\alpha+1)(\alpha+1)}. \end{align}

By making $\Phi = v^{\alpha+1}$, then
\begin{align*}
\mathcal{E}=\frac{\omega}{2}\Phi_s^2(s)+U(\Phi(s))=\frac{\alpha+1}{2\alpha+1}(v')^2v^{2\alpha}+\frac{\alpha+1}{2\alpha}v^{2\alpha}-\frac{\alpha+1}{2\alpha+1}v^{2\alpha+1} \end{align*}
and
$$
U_{max}= \frac{1}{2\alpha}\frac{\alpha+1}{2\alpha+1}.
$$
Consider the function $f$ given for
\begin{align}\label{function}
f(x)=\sqrt{\frac{1}{kx^k}+x-\left(1+\frac{1}{k}\right)},
\end{align}
where $k=\frac{(p-1)s_c}{2}=2\alpha$. Hence, if $v_s(0)$ satisfies the condition
$$
v_s(0)<
\left\{
\begin{array}{ll}
+f(v(0)),\,\,\,&\mbox{if }v(0)<1,\\
-f(v(0)),\,\,\,&\mbox{if }v(0)\geq 1,
\end{array}
\right.
$$
then $\Phi=v^{\alpha+1}$ satisfies the conditions of Proposition \ref{proplush}. Indeed, the condition $\mathcal{E}<U_{max}$ is equivalent to
$$2\alpha(v')^2v^{2\alpha}+(2\alpha+1)v^{2\alpha}-2\alpha v^{2\alpha+1}<1$$
that is,
$$|v_s|<f(v).$$
Hence, the condition (A) means
$$v(0)<1\,\,\,\,\,\mbox{ and }\,\,\,\,\,-f(v(0))<v_s(0)<f(v(0))$$
and the condition (B) holds if and only if
$$|v_s(0)|>f(v(0))\,\,\,\,\,\mbox{ and }\,\,\,\,\,v_s(0)<0.$$
More precisely,
$$v_s(0)<-f(v(0))$$
and the condition (C) is equivalent to
$$v(0)<1\,\,\,\,\,\mbox{ and }\,\,\,\,\,v_s(0)=-f(v(0)).$$

Therefore, from \eqref{B}, \eqref{lamandmu} and from the definition of $v$, we have
\begin{align*}
V(0)&=(\mu \Phi(\lambda t))^{\frac{1}{\alpha+1}}\Bigg|_{t=0}=\mu^{\frac{8}{N(p-1)+4+2b}}v\left(\frac{8\sqrt 2}{\sqrt{Ns_c}}\frac{E[u_0]}{M[u_0]}t\right)\Bigg|_{t=0}\\
&=\mu^{\frac{8}{N(p-1)+4+2b}}v(0)=\frac{Ns_cM^2}{4E[u_0]}v(0)
\end{align*}
and
$$V_t(0)=\mu^{\frac{8}{N(p-1)+4+2b}}\frac{8\sqrt 2}{\sqrt{Ns_c}}\frac{E[u_0]}{M[u_0]}v_s(0)=\frac{Ns_cM^2}{4E[u_0]}\frac{8\sqrt 2}{\sqrt{Ns_c}}\frac{E[u_0]}{M[u_0]}v_s(0)=M[u_0]\sqrt{8Ns_c}v_s(0).$$
Furthermore,
\begin{align*}
\frac{V_t(0)}{M[u_0]}=\sqrt{8Ns_c}v_s(0)<\sqrt{8Ns_c}g(v(0))=\sqrt{8Ns_c}g\left(\frac{4}{Ns_c}\frac{V(0)E[u_0]}{M[u_0]^2}\right),
\end{align*}
which completes the proof of Theorem \ref{lush1}.
\end{proof}
We now proceed to the proof of Theorem \ref{lush2}. For that, we consider the following proposition.
\begin{prop}\label{proplush2} Let $p>1$ and $N\geq1$. Then, the following inequality
\begin{equation}\label{interp}
\|u\|^{2}_{L^2}\leq C_{p,N}\left(\|xu\|_{L^2}^{\frac{N(p-1)+2b}{2}}\||\cdot|^{\frac{-b}{p+1}}u\|^{p+1}_{L^{p+1}}\right)^{\frac{2}{N(p-1)+2(p+1)+2b}}
\end{equation}
holds with the sharp constant $C_{p,N}$ (depending on the nonlinearity $p$ and dimension $N$) given by \eqref{cnstlush}. Moreover, the equality occurs if and only if there exists $\beta\geq 0$, $\alpha\leq 0$ such that $|u(x)|=\beta\phi(\alpha x)$, where
{$$
\phi(x)=
\left\{
\begin{array}{ll}
|x|^{\frac{b}{p-1}}(1-|x|^2)^{\frac{1}{p-1}}&\mbox{ if }0\leq |x|<1,\\
0&\mbox{ if } |x|>1.
\end{array}
\right.
$$}
\end{prop}
The proof of Proposition \ref{proplush2} follows the ideas of \cite{DR_Going}.
\begin{proof}
Let $R>0$ to be specified later. Split the mass of $u$ as follows
\begin{equation}
\int |u(x)|^2\,dx=\frac{1}{R^2}\int_{|x|\leq R}(R^2-|x|^2)|u(x)|^2\,dx+\frac{1}{R^2}\int_{|x|\leq R}|x|^2|u(x)|^2\,dx+\int_{|x|\geq R}|u(x)|^2\,dx.
\end{equation}
By Hölder inequality we have
\begin{align}\label{propdes1}
\frac{1}{R^2}\int(R^2-|x|^2)|u(x)|^2\,dx&\leq \frac{1}{R^2}\left(\int_{|x|\leq R}|x|^{\frac{2b}{p-1}}(R^2-|x|^2)^{\frac{p+1}{p-1}}\,dx\right)^{\frac{p-1}{p+1}}\left(\int |x|^{-b}|u(x)|^{p+1}\,dx\right)^{\frac{2}{p+1}}\nonumber\\
&\leq \frac{1}{R^2}\left(\int_{|x|\leq 1}R^{\frac{2b}{p-1}}|y|^{\frac{2b}{p-1}}(R^2-R^2|y|^2)^{\frac{p+1}{p-1}}R^N\,dy\right)^{\frac{p-1}{p+1}}\left(\int |x|^{-b}|u(x)|^{p+1}\,dx\right)^{\frac{2}{p+1}}\nonumber\\
&=R^{\frac{N(p-1)+2b}{p+1}}D_{p,N} \left\||\cdot|^{-\frac{b}{p+1}}u\right\|^{2}_{p+1},
\end{align}
where
$$D_{p,N}=\left(\int_{|y|\leq 1}|y|^{\frac{2b}{p-1}}(1-|y|^2)^{\frac{p+1}{p-1}}\,dy\right)^{\frac{p-1}{p+1}}.$$
Furthermore,
\begin{equation}\label{propdes2}
\frac{1}{R^2}\int_{|x|\leq R}|x|^2|u(x)|^{2}\,dx+\int_{|x|\geq R}|u(x)|^2\,dx\leq \frac{1}{R^2}\int |x|^2|u(x)|^2\,dx.
\end{equation}
Combining \eqref{propdes1} and \eqref{propdes2}, we get
\begin{equation}\label{ineq}
\forall R>0,\,\,\,\, \|u\|^{2}_{L^2}\leq D_{p,N}\left\||\cdot|^{-\frac{b}{p+1}}u\right\|_{L^{p+1}}^2R^{\frac{N(p-1)+2b}{p+1}}+\frac{1}{R^2}\|xu\|^{2}_{L^2}.
\end{equation}
Let $F:(0,+\infty)\to \Real$ given by $F(R)=AR^{\alpha}+BR^{-2}$, where $A,B>0\mbox{ and }\alpha>0$. The minimum value of $F$ is reached at $R=\left(\frac{2B}{\alpha A}\right)^{\frac{1}{\alpha+2}}$ and
$$F\left(\left(\frac{2B}{\alpha A}\right)^{\frac{1}{\alpha+2}}\right)=A\left(\frac{2B}{\alpha A}\right)^{\frac{\alpha}{\alpha+2}}+B\left(\frac{\alpha A}{2B}\right)^{\frac{2}{\alpha+2}}=\frac{2+\alpha}{\alpha}(\alpha A)^{\frac{2}{\alpha+2}}(2B)^{\frac{\alpha}{\alpha+2}}.$$
Thus, by taking
$$R=\left(\frac{p+1}{N(p-1)+2b}\frac{2\|xu\|^2_{L^2}}{D_{p,N}\left\||\cdot|^{-\frac{b}{p+1}}u\right\|_{L^{p+1}}^2}\right)^{\frac{p+1}{N(p-1)+2(p+1)+2b}}$$
 in \eqref{ineq}, we have
 $$\|u\|^{2}_{L^2}\leq C_{p,N}^2\left\||\cdot|^{-\frac{b}{p+1}}u\right\|_{L^{p+1}}^{\frac{4(p+1)}{N(p-1)+2(p+1)+2b}}\|xu\|_{L^2}^{\frac{2N(p-1)+4b}{N(p-1)+2(p+1)+2b}}$$
 where
\begin{equation}\label{cnstlush}
C_{p,N}=\left(\frac{N(p-1)+2(p+1)+2b}{2N(p-1)+4b}\right)^{\frac{1}{2}}\left(\frac{N(p-1)+2b}{p+1}D_{p,N}\right)^{\frac{(p+1)}{N(p-1)+2(p+1)+2b}}2^{\frac{N(p-1)+2b}{2N(p-1)+4(p+1)+4b}}.
\end{equation}

Note that equality in \eqref{interp} holds if and only if there exists $R > 0$ such that \eqref{ineq} is an equality. This
is equivalent to the fact that for some $R > 0$, both \eqref{propdes1} and \eqref{propdes2} are equalities. The inequality
\eqref{propdes1} is an equality if and only if, for $|x| < R$, $|x|^{-b}|u(x)|^{p+1} = c|x|^{\frac{2b}{p-1}}(R^2-|x|^2)^{\frac{p+1}{p-1}}$ for some constant $c \geq 0$, and inequality \eqref{propdes2} is an equality if and only if $u(x) = 0$ for $|x| \geq R$. This completes the proof of Proposition \ref{proplush2}.
\end{proof}
\begin{proof}[Proof of Theorem \ref{lush2}] Since the energy
is$$E[u_0]=\frac{1}{2}\|\nabla u(t)\|^{2}_{L^2}-\frac{1}{p+1}\left\||\cdot|^{-\frac{b}{p+1}}u(t)\right\|^{p+1}_{L^{p+1}},$$
from \eqref{d2var}, we obtain
\begin{align*}
V_{tt}(t)&=4(N(p-1)+2b)E[u_0]-2(N(p-1)+2b-4)\|\nabla u(t)\|^{2}_{L^2(\Real^N)}\\&=16E[u_0]-\frac{8(p-1)s_c}{p+1}\left\||\cdot|^{-\frac{b}{p+1}}u(t)\right\|^{p+1}_{L^{p+1}}.
\end{align*}
Using the sharp interpolation inequality \eqref{interp}
\begin{equation}\label{inlush}
V_{tt}(t)\leq 16E[u_0]-\frac{8(p-1)s_c}{(p+1)(C_{p,N})^{\frac{N(p-1)}{2}+(p+1)+b}}\frac{M[u_0]^{\frac{N(p-1)}{4}+\frac{(p+1)}{2}+\frac{b}{2}}}{V(t)^\frac{N(p-1)+2b}{4}},
\end{equation}
with $C_{p,N}$ from \eqref{interp}. As done in the proof of Proposition \ref{lush1}, take $v(s)$ with $s=at$ such that
$$V(t)=\mu v(\lambda t),\,\,\,\,\lambda=\sqrt{\frac{32E[u_0]}{\mu}},$$
where
$$\mu=\left(\frac{s_c(p-1)}{2(p+1)}\right)^{\frac{4}{N(p-1)+2b}}\frac{M[u_0]^{1+(p+1)\left(\frac{2}{N(p-1)+2b}\right)}}{(C_{p,N})^{2+(p+1)\left(\frac{4}{N(p-1)+2b}\right)}E[u_0]^{\frac{4}{N(p-1)+2b}}}.$$
Hence, applying in the inequality \eqref{inlush}, we have
$$v_{ss}(s)\leq \frac{1}{2}\left(1-v^{-\frac{N(p-1)+2b}{4}}(s)\right).$$
If the inequality in the above expression is replaced by an equality, then we have that the following energy is conserved$$
\mathcal{E}(s)=\frac{k}{1+k}\left((v(s))^2-v(s)-\frac{1}{kv(s)^k}\right),
$$
where as before $k=\frac{(p-1)s_c}{2}=\frac{N(p-1)+2b}{4}-1$. The maximum of the function $$f(x)=\frac{k}{1+k}\left(x+\frac{1}{kx^k}\right),$$ attained at $x=1$, is $-1$. As we did to (A), (B) and (C), we identify the three sufficient conditions for blow-up in finite time.
\item[($A^*$)] $\mathcal{E}(0)<-1$ and $v(0)<1,$
\item[($B^*$)] $\mathcal{E}(0)>-1$ and $v_s(0)<0,$
\item[($C^*$)] $\mathcal{E}(0)=-1$, $v_s(0)<0$ and $v(0)<1.$

If $v_s(0)$ satisfies the condition
$$
v_s(0)<
\left\{
\begin{array}{ll}
+f(v(0)),\,\,\,&\mbox{if }v(0)<1\\
-f(v(0)),\,\,\,&\mbox{if }v(0)\geq 1,
\end{array}
\right.
$$
then $v$ satisfies one of the conditions (A*), (B*) and (C*). Indeed, recalling the function $f$ from \eqref{function} and using the definition of $\mathcal{E}$, we obtain
\item[a)] $\mathcal{E}<-1$ if and only if $|v_s|<f(v).$
\item[b)] $\mathcal{E}\geq-1$ if and only if $|v_s|\geq f(v).$\\
Then the previous conditions can be written in the following form:

$(A^*) \Leftrightarrow v(0)<1\mbox{ and }-f(v(0))<v_s(0)<f(v(0)),$

$(B^*) \Leftrightarrow v_s(0)<-f(v(0))$

$(C^*) \Leftrightarrow v_s(0)=-f(v(0)),\,\,\,v(0)<1.$\\
Substituting back $V(t)$, we obtain
$$\frac{V_t(0)}{\lambda\mu}<g\left(\frac{V(0)}{\mu}\right),$$
where $g$ is defined in \eqref{g}. Hence,

$$\frac{V_t(0)}{4\sqrt2}\cdot \left(\frac{2(p+1)}{s_c(p-1)}(C_{p,N})^{\frac{N(p-1)+2b}{2}+(p+1)}\right)^{\frac{2}{N(p-1)+2b}}\frac{(C_{p,N})^{1+(p+1)\left(\frac{2}{N(p-1)+2b}\right)}}{E[u_0]^{\frac{s_c}{N}}M[u_0]^{\frac{1}{2}+(p+1)\left(\frac{1}{N(p-1)+2b}\right)}}<g(\theta),$$
with
$$\theta =\left(\frac{2(p+1)}{s_c(p-1)}(C_{p,N})^{\frac{N(p-1)+2b}{2}+(p+1)}\right)^{\frac{4}{N(p-1)+2b}}\frac{E[u_0]^{\frac{4}{N(p-1)+2b}}}{M[u_0]^{1+(p+1)\left(\frac{2}{N(p-1)+2b}\right)}}V(0).$$
This completes the proof of Theorem \ref{lush2}.
\end{proof}